\documentclass[a4paper]{amsart}
\usepackage[utf8]{inputenc}
\usepackage[english]{babel}
\usepackage{amsfonts}
\usepackage{amssymb,amsthm}
\usepackage[left=3cm, right=3cm]{geometry}

\usepackage{mathtools}
\mathtoolsset{showonlyrefs}

\usepackage{hyperref}
\usepackage{xcolor}
\hypersetup{
    colorlinks,
    linkcolor={blue!90!black},
    citecolor={green!50!black},
    urlcolor={blue!90!black}
}

\makeatletter
\def\blfootnote{\gdef\@thefnmark{}\@footnotetext}
\makeatother

\newcommand{\R}{\mathbb R}
\newcommand{\C}{\mathbb C}
\renewcommand{\P}{\mathbb P}
\newcommand{\mc}{\mathcal}

\renewcommand{\phi}{\varphi}
\newcommand{\n}{\nabla}
\newcommand{\pd}{\partial}
\newcommand{\tr}{{\rm tr}}

\renewcommand{\ge}{\geqslant}
\renewcommand{\le}{\leqslant}

\renewcommand{\bar}[1]{\mkern 0.5mu\overline{\mkern-0.5mu#1\mkern-0.5mu}\mkern 0.5mu}

\theoremstyle{plain}
\newtheorem{theorem}{Theorem}[section]
\newtheorem*{theorem*}{Theorem}
\newtheorem{lemma}[theorem]{Lemma}
\newtheorem{corollary}[theorem]{Corollary}
\newtheorem{proposition}[theorem]{Proposition}

\theoremstyle{definition}
\newtheorem{remark}[theorem]{Remark}

\newtheorem{definition}[theorem]{Definition}
\newtheorem{example}[theorem]{Example}

\begin{document}

\date{}
\date{}

\title[The HCF on manifolds with non-negative Griffiths curvature]{The Hermitian curvature flow on manifolds with non-negative Griffiths curvature}
\author{Yury Ustinovskiy}

\address{Department of Mathematics, Princeton University}
\email{yuryu@math.princeton.edu}

\begin{abstract}
In this paper we study a particular version of the \emph{Hermitian curvature flow} (HCF) over a compact complex Hermitian manifold $(M,g,J)$. We prove that if the initial metric has Griffiths positive (non-negative) Chern curvature $\Omega$, then this property is preserved along the flow. On a manifold with Griffiths non-negative Chern curvature the HCF has nice regularization properties, in particular, for any $t>0$ the zero set of $\Omega(\xi,\bar\xi,\eta,\bar\eta)$ becomes invariant under certain \emph{torsion-twisted} parallel transport.
\end{abstract}
\maketitle

\blfootnote{keywords: \emph{Hermitian curvature flow}, \emph{Griffiths positivity}, \emph{strong maximum principle}}

\section*{Introduction}
The K\"ahler-Ricci flow has been demonstrated to be a powerful tool in various geometrical classification problems in K\"ahler geometry, see, e.g.,~\cite{gu-09, mo-88}. However, on a general Hermitian manifold Ricci flow does not preserve Hermitian condition. This observation motivated Streets and Tian~\cite{st-ti-11} to introduce a family of \emph{Hermitian curvature flows} (HCF) on an arbitrary Hermitian manifold $(M,g,J)$.
One uses Chern connection $\n$ and its curvature $\Omega$ to define the HCF.  Under this flow metric is evolved according to the equation
\begin{equation}\label{eq:HCF_intro}
\frac{dg}{dt}=-S-Q,
\end{equation}
where $S_{k\bar l}=g^{i\bar j}\Omega_{i\bar jk\bar l}$ is the \emph{second Chern-Ricci curvature} of Chern connection and $Q$ is an arbitrary quadratic torsion term of type (1,1). In~\cite{st-ti-11} the authors prove short time existence for this flow and derive basic long time blowup and regularity results.

In this paper we study flow~\eqref{eq:HCF_intro} with the specific torsion term $Q_{i\bar j}=\frac{1}{2}g^{m\bar n}g^{p\bar s}T_{pm\bar j}T_{\bar s\bar n i}$ and in what follows this flow is referred to as the HCF. It turns out that for this particular choice of $Q$ (with a special choice of connection on the bundle of curvature-type tensors) the evolution equation for curvature is significantly simplified. Then the evolution equation for $\Omega$ becomes very similar to the curvature evolution under the standard Ricci flow (Proposition~\ref{prop:dt_Omega_torsion_twisted}). This observation allows us to mimic in the case of the HCF the proof of the positivity preservation properties of K\"ahler-Ricci flow. Namely, we adopt arguments of Bando~\cite{ba-84} and Mok~\cite{mo-88}, who proved that under  K\"ahler-Ricci flow the positivity of the bisectional holomorphic curvature is preserved, and prove that the HCF preserves Griffiths positivity (non-negativity) of the Chern curvature.

\begin{theorem}
Let $g(t), t\in[0,\tau)$ be the solution to the HCF on a compact complex Hermitian manifold $(M,g_0,J)$. Assume that the Chern curvature $\Omega^{g_0}$ at the initial moment $t=0$ is Griffiths non-negative (resp. positive), i.e., for $\xi,\eta\in T^{1,0}M$:
\[
\Omega^{g_0}(\xi,\bar\xi,\eta,\bar\eta)\ge 0\quad (\mbox{resp. }>0).
\]
Then for $t\in[0,\tau)$ the Chern curvature $\Omega(t)=\Omega^{g(t)}$ remains Griffiths non-negative (resp. positive).

If, moreover, the Chern curvature $\Omega^{g_0}$ is Griffiths positive at least at one point $x\in M$, then for any $t\in(0;\tau)$ the Chern curvature is Griffiths positive everywhere on $M$.
\end{theorem}

We also use the argument of Brendle and Schoen~\cite{br-sc-08} to prove some regularization properties of the HCF on the manifolds with Griffiths non-negative Chern curvature. Namely, we prove that for any $t>0$ the Chern curvature tensor $\Omega(t)=\Omega^{g(t)}$ of the solution $g(t)$ to the HCF acquires additional invariance properties, which are not a priori satisfied  at the initial moment.

\begin{theorem}
Let $g(t), t\in[0,\tau)$ be a solution to the HCF on a compact complex Hermitian manifold $(M,g_0,J)$. Assume that the Chern curvature $\Omega^{g_0}$ at the initial moment $t=0$ is Griffiths non-negative. Then for any $t_0\in(0,\tau)$ the set
\[
Z=\{(\xi,\eta) |\ \xi,\eta\in T^{1,0}M,\ \Omega^{g(t_0)}(\xi,\bar\xi,\eta,\bar\eta)=0\}\subset T^{1,0}M\oplus T^{1,0}M
\]
is invariant under the torsion-twisted parallel transport
\begin{equation}
\begin{cases}
\n_{\gamma'}\xi=T(\gamma',\xi), \\
\n_{\gamma'}\eta=-g(\eta,T(\gamma',\cdot))^\#,
\end{cases}
\end{equation}
where $\#\colon T^*M\to TM$ is the isomorphism induced by $g(t_0)$.
\end{theorem}

We expect, that preservation of Griffiths positivity will imply strong existence results for the HCF. Together with the regularization properties of the HCF this might allow to prove certain uniformization theorems for Hermitian manifolds with non-negative Griffiths curvature. For example, with the use of the strong maximum principle for Griffiths positivity we are able to infer the following result.
\begin{proposition}
Let $(M,g_0,J)$ be a compact complex $n$-dimensional Hermitian manifold such that
\begin{enumerate}
\item its Chern curvature $\Omega$ is Griffiths non-negative;
\item $\Omega$ is strictly positive at some point $x_0\in M$.
\end{enumerate}
Then $M$ is biholomorphic to the projective space $\P^n$.
\end{proposition}

Besides the HCF, there exist several other natural flows on a general complex Hermitian manifold. For example, in~\cite{gi-11} Gill studies \emph{Chern-Ricci flow}~--- evolution of a Hermitian metric along the first Chern-Ricci curvature $g^{k\bar l}\Omega_{i\bar jk\bar l}$. Like K\"ahler-Ricci flow this flow reduces to a scalar parabolic flow of Monge-Amp\`ere type. In~\cite{to-we-15} Tosatti and Weinkove give characterization of maximal existence time for this flow in all dimensions and analyse convergence in complex dimension two. Liu and Yang in~\cite{li-ya-12} propose to use a mixed linear combination of various types of Ricci curvatures on a Hermitian manifold, including the curvature of the Bismut connection. Recently in~\cite{ya-16} Yang explicitly computed evolution of the standard \emph{round} metric on a Hopf manifold $S^{2n-1}\times S^1$ under the Chern-Ricci flow. This computation demonstrates that Chern-Ricci flow does not preserve Griffiths non-negativity. It would be interesting to investigate whether other natural Hermitian flows on a general complex Hermitian manifold or the HCF with a different choice of $Q$ preserve some positivity properties.

The rest of the paper is organized as follows. In Section~\ref{sec:background} we fix notations and give basic background information about the Chern connection and its curvature. Sections~\ref{sec:variation-formulas} and~\ref{sec:evolution} contain the major part of the computations. We also define the crucial notion of torsion-twisted connection on the space of curvature tensors in Section~\ref{sec:evolution}. In Section~\ref{sec:positivity_properties} positivity results for the zeroes of a Griffiths non-negative curvature tensors are proved. These inequalities are the straightforward generalizations of the corresponding results in the K\"ahler situation. In Section~\ref{sec:main} we prove our main results~--- weak and strong maximum principles for the Griffiths positivity along the HCF. Finally, in Section~\ref{sec:applications} we give some geometric applications these results.
\section*{Acknowledgements}

I am grateful to my advisor Gang Tian for bringing the problems related to the Hermitian curvature flow to my attention and for numerous helpful discussions.

\section{Background}\label{sec:background}
Let $(M,g,J)$ be a compact Hermitian complex manifold, where $J\colon TM\to TM$ is an operator of almost complex structure, $g\colon TM\otimes TM\to \R$ is a $J$-invariant Riemannian metric. We extend $g$ to a $\C$-bilinear form on the complexified tangent bundle $T_\C M:=TM\otimes\C$ which is denoted by the same symbol. Finally let $T_\C M=T^{1,0}M\oplus T^{0,1}M$ be the decomposition into the $\pm i$ eigenspaces of $J$. Any Riemannian metric $g$ defines the corresponding \emph{Hermitian} metric $g(\xi,\bar\eta)$ on $T^{1,0}M$. 

From now on, Greek letters $\xi,\eta,\zeta,\dots$ denote \emph{complex} (1,0)-vectors and vector fields.

\begin{definition}There is a unique connection $\n\colon \Gamma(TM)\to \Gamma(TM\otimes \Lambda^1M)$ such that:
\begin{enumerate}
\item $\n$ preserves the Riemannian metric, i.e., $\n g =0$;
\item $\n$ preserves the complex structure, i.e., $\n J=0$;
\item torsion tensor $T(X,Y):=\n_XY-\n_YX-[X,Y]$ satisfies $T(JX,Y)=T(X,JY)$ for any $X,Y\in TM$. Equivalently $T(\xi,\bar{\eta})=0$ for the natural extension of the torsion tensor to $T_\C M$ and $\xi,\eta\in T^{1,0}M$. In other words, the complex (1,1)-part of $T$ vanishes.
\end{enumerate}
This connection is called the \emph{Chern connection}. 
\end{definition}
Throughout this paper we automatically extend $\n$ and all other connections on $TM$ to $\C$-linear connections on $T_\C M$ and to connections on all associated vector bundles, e.g., $T^*M, \Lambda^k TM$, $\mathrm{End}(TM)$, etc.

\begin{remark}
Alternatively one can define the Chern connection as the unique Hermitian connection on the holomorphic vector bundle $(\mc E,h)=(T^{1,0}M,g)$ compatible with the holomorphic structure (see, e.g.,~\cite[Prop.\,10.2]{ko-no-69}), i.e., $\n\colon \Gamma(T^{1,0}M)\to \Gamma(T^{1,0}M\otimes \Lambda_\C^1M)$ such that:
\begin{enumerate}
\item $\n$ preserves Hermitian metric, $\n g=0$;
\item the (0,1)-part of $\n$ is $\bar\pd$, i.e., $\n_{\bar \xi}\eta=0$ for any local \emph{holomorphic} vector field $\eta\in\Gamma(T^{1,0}M)$ and any (0,1)-vector $\bar \xi\in T^{0,1}M$.
\end{enumerate}
\end{remark}

\begin{definition}
The \emph{Chern curvature} of a Hermitian manifold $(M, g, J)$ is the curvature of $\n$.
\[
\Omega(X,Y)Z:=(\n_X\n_Y-\n_Y\n_X-\n_{[X,Y]})Z,
\]
where $X,Y,Z\in TM$. We also define a tensor with 4 vector arguments by lowering one index.
\[
\Omega(X,Y,Z,W):=g\bigl((\n_X\n_Y-\n_Y\n_X-\n_{[X,Y]})Z, W\bigr).
\]
If there might be an ambiguity of which metric is used to define curvature/torsion tensors, we use the corresponding superscript $\Omega^g$, $T^g$.
\end{definition}

Tensor $\Omega$ satisfies a number of symmetries.
\begin{proposition}\label{prop:curvature_symmetries}
For any real vectors $X,Y,Z,W\in TM$ one has
\begin{itemize}
\item $\Omega(X,Y,Z,W)=-\Omega(Y,X,Z,W)$, $\Omega(X,Y,Z,W)=-\Omega(X,Y,W,Z)$;
\item $\Omega(JX,JY,Z,W)=\Omega(X,Y,Z,W)$, $\Omega(X,Y,JZ,JW)=\Omega(X,Y,Z,W)$.
\end{itemize}
\end{proposition}
Symmetries of $\Omega$ imply that for $\xi,\eta,\zeta,\nu\in T^{1,0}M$
\[
\Omega(\xi,\eta,\cdot,\cdot)=\Omega(\cdot,\cdot,\zeta,\nu)=0,\quad
\Omega(\xi,\bar\eta,\zeta,\bar{\nu})=\bar{\Omega(\eta,\bar\xi,\nu,\bar{\zeta})},
\]
in particular $\Omega(\xi,\bar \xi,\eta,\bar\eta)\in\R$. It is easy to check that the values $\Omega(\xi,\bar \xi,\eta,\bar\eta)$ for $\xi,\eta \in T^{1,0}M$ completely determine tensor $\Omega$.
\begin{definition}\label{def:curvature_type}
We say that a real tensor $u\in (T^* M)^{\otimes 4}$ is of \emph{curvature type}, if it satisfies all the symmetries and $J$-invariance properties of Proposition~\ref{prop:curvature_symmetries}.
\end{definition}

In computations we will use coordinate notation for different tensors and assume Einstein summation for repeated upper/lower indices, e.g.,
\[
\begin{split}
&\n_i:=\n_{\pd/\pd z_i},\\
&\Omega_{i\bar j k\bar l}:=\Omega(\pd/\pd z_i,\pd/\pd \bar z_j,\pd/\pd z_k,\pd/\pd \bar z_l),\\ 
&T_{ij}^k\ \pd/ \pd z_k:=T(\pd/\pd z_i,\pd/\pd z_j),\\
&T_{ij\bar l}:=g(T(\pd/\pd z_i,\pd/\pd z_j), \pd /\pd\bar{z}_l).
\end{split}
\]
There is no summation assumed if two lower or two upper indices are repeated.

Unlike the Riemannian case, the Chern curvature does not satisfy the classical Bianchi identities, since the Chern connection has torsion. However, in this case slightly modified identities, involving torsion still hold~\cite[Ch.\,III, Thm.\,5.3]{ko-no-63}.

\begin{proposition}[Bianchi identities for the Chern curvature]
For any complex vectors $X,Y,Z\in T_\C M$ one has respectively the first and the second Bianchi identities
\begin{equation}
\begin{split}
&\sum_{\mathfrak{S}_3} \Omega(X,Y)Z=\sum_{\mathfrak{S}_3} \bigr(T(T(X,Y), Z)+\n_X T(Y,Z)\bigl),\\
&\sum_{\mathfrak{S}_3}\bigl(\n_X\Omega(Y,Z)+\Omega(T(X,Y),Z)\bigr)=0,
\end{split}
\end{equation}
where the sum is taken over all cyclic permutations. In particular, using the $J$-invariance of $\Omega$ and the vanishing of the $(1,1)$-part of $T$, one gets that for any (1,0)-complex vectors $\xi,\eta,\zeta\in T^{1,0}M$
\begin{equation*}
\begin{split}
	\Omega(\xi, \bar \eta)\zeta-\Omega(\zeta, \bar \eta)\xi=&\n_{\bar \eta}T(\zeta,\xi),\\
	\n_\zeta\Omega(\xi,\bar\eta)-\n_\xi\Omega(\zeta,\bar\eta)=&\Omega(T(\xi,\zeta),\bar\eta).
\end{split}
\end{equation*}
Equivalently, in the coordinates
\begin{equation}
	\begin{split}
	\Omega_{i \bar{j} k \bar{l}}=\Omega_{k \bar{j} i \bar{l}} + \n_{\bar{j}} T_{ki\bar l},\quad &
	\Omega_{i \bar{j} k \bar{l}}=\Omega_{i \bar{l} k \bar{j}} + \n_{i} T_{\bar{l}\bar{j} k},\\
	\n_m \Omega_{i \bar{j} k \bar{l}} = \n_i\Omega_{m \bar{j} k \bar{l}} + T^p_{i m}\Omega_{p \bar{j} k \bar{l}},\quad &
	\n_{\bar n} \Omega_{i \bar{j} k \bar{l}} = \n_{\bar j}\Omega_{i \bar{n} k \bar{l}} + T^{\bar s}_{\bar j \bar n}\Omega_{i \bar{s} k \bar{l}}.
	\end{split}
\end{equation}
\end{proposition}

\begin{definition}\label{def:positive}
For a curvature-type tensor $u$ we write $u>0$, resp.\,$u\ge 0$, if for any non-zero $\xi,\eta\in T^{1,0}M$ tensor $u$ satisfies
\[
u(\xi,\bar \xi, \eta,\bar \eta)>0, (\mbox{resp.} \ge0).
\]
We say that a complex Hermitian manifold $(M,g,J)$ has \emph{Griffiths positive} (resp.\,non-negative) curvature, if its Chern curvature is positive (resp.\,non-negative), i.e., $\Omega>0$, (resp.\,$\Omega\ge 0$). 
\end{definition}

\begin{remark}
The notion of Griffiths-positivity could be defined for an arbitrary holomorphic Hermitian bundle. Namely, holomorphic vector bundle $(\mc E, h)$ is said to be Griffiths-positive, if its Chern curvature tensor $\Omega\in\Lambda^{1,1}T^*M\otimes {\rm End}(\mc E)\simeq \Lambda^{1,1}T^*M\otimes \mc E^*\otimes \mc E^*$ is positive on all non-zero decomposable tensors $\xi\otimes\bar\xi\otimes v\otimes\bar v$, $\xi\in T^{1,0}M$, $v\in \mc E$.

Griffiths positivity implies ampleness of $\mc E$ and conjecturally any ample bundle $\mc E$ admits Griffiths positive metric $h$, see~\cite[Prob.\,11.14]{de}.

Complex Hermitian manifold $(M,g,J)$ has \emph{Griffiths positive} curvature if and only if its holomorphic tangent bundle $(T^{1,0}M, g)$ is positive in the sense of Griffiths, see~\cite{de}.
\end{remark}

\begin{remark}
If a Hermitian manifold $(M,g,J)$ is K\"ahler, i.e., its Chern connection does not have torsion, then the curvature tensor $\Omega(X,JX,JY,Y)$ coincides with the bisectional holomorphic curvature. In particular, in this case $(M,g,J)$ has Griffiths positive curvature if and only if the \emph{bisectional holomorphic curvature} of the K\"ahler metric $g$ is positive.
\end{remark}

\begin{example}\label{ex:homogeneous}
It is well known, that the induced metric on a quotient bundle of a Griffiths non-negative Hermitian bundle $(\mc E,h)$ is Griffiths non-negative,~\cite[Prop.\,6.10]{de}. Hence, if $T^{1,0}M$ is globally generated, then the metric $h$, induced on $T^{1,0}M$ by the natural projection from the trivial bundle $H^0(M,T^{1,0} M)\simeq \C^r$ with a constant metric, is Griffiths non-negative. 

In particular, any complex homogeneous space has a metric with Griffiths non-negative Chern curvature. Note, that this metric does not have to be K\"ahler, even though the underlying manifold might be projective. In fact, from the generalized Frankel conjecture proved by Mok~\cite{mo-88}, it follows, that Griffiths positive K\"ahler metric exists only on the finite quotients of rational symmetric spaces. For example the flag manifold $U(n)/T^n, n\ge 3$ is projective, admits Hermitian metric of non-negative Griffiths curvature, but does not admit a K\"ahler metric of non-negative bisectional holomorphic curvature. 
\end{example}

\section{Variation formulas}\label{sec:variation-formulas}
Let $\delta g=k$ be a variation of the Hermitian metric. In this section we compute variations of the Chern connection $\n$, its torsion and the Chern curvature $\Omega$. Note that unlike $\n$ itself, its variation $\delta\n$ is a tensor.
\begin{proposition}\label{prop:var_nabla}
Under the variation of the metric $\delta g=k$, the variation $\delta\n$ of the Chern connection is given by the formula
\begin{equation}
\begin{split}
(\delta\n)_{\bar\xi}\eta=(\delta\n)_{\xi}\bar\eta=0,\quad 
g\bigl((\delta\n)_\xi \eta, \bar\zeta\bigr)=\n_\xi k(\eta,\bar\zeta).
\end{split}
\end{equation}
\end{proposition}

\begin{proof}
To prove the first formula, we just notice that $\n_{\bar\xi}\eta=\bar\pd_{\bar\xi}\zeta$ is completely defined by the holomorphic structure on the bundle $T^{1,0}M$ and is independent of the choice of $g$.

To prove the second formula, let us take the variation of the identity
\[
\xi\!\cdot\! g(\eta, \bar\zeta)=g(\n_\xi \eta,\bar\zeta)+g(\eta,\n_\xi\bar\zeta),
\]
where $\xi,\eta,\zeta\in T^{1,0}M$. 
\[
\xi\!\cdot\! k(\eta,\bar\zeta)=k(\n_\xi \eta,\bar\zeta)+k(\eta,\n_{\xi}\bar\zeta)+g((\delta\n)_\xi \eta,\bar\zeta).
\]
Collecting the expressions involving $k$ on the one side, we get the desired identity. 
\end{proof}
Proposition~\ref{prop:var_nabla} immediately imply the variation formula for the torsion. 
\begin{proposition}\label{prop:var_torsion}
With the same notations as in Proposition~\ref{prop:var_nabla}, the variation of the torsion tensor $T(\xi,\eta)$ is given by the formula
\[
g((\delta T)(\xi,\eta),\bar\zeta)=\n_\xi k(\eta,\bar\zeta)-\n_\eta k(\xi, \bar\zeta).
\]
\end{proposition}

\begin{proposition}\label{prop:var_omega}
With the same notations as in Proposition~\ref{prop:var_nabla}, the variations of the (3,1)-type and (4,0)-type Chern curvatures are given by the formulas
\begin{equation}
\begin{split}
g\bigl((\delta\Omega)(\xi,\bar\eta)\zeta,\bar\nu\bigr)&=
-\n_{\bar \eta}\n_\xi k(\zeta,\bar\nu)+\n_{\n_{\bar{\eta}\xi}}k(\zeta,\bar{\nu}),\\
(\delta\Omega)(\xi,\bar\eta,\zeta,\bar\nu)&=
k\bigl(\Omega(\xi,\bar\eta)\zeta,\bar\nu\bigr)-\n_{\bar \eta}\n_\xi k(\zeta,\bar\nu)+\n_{\n_{\bar{\eta}\xi}}k(\zeta,\bar{\nu}).
\end{split}
\end{equation}
\end{proposition}
\begin{proof}
Clearly the second formula follows from the first one. Before we start proving the first formula, note that for any $\zeta,\nu\in \Gamma(T^{1,0}M)$ we have $[\zeta,\bar{\nu}]=\n_{\zeta}\bar{\nu}-\n_{\bar{\nu}}\zeta$, since the $(1,1)$-part of the torsion vanishes. In particular, the $(1,0)$-part of $[\zeta,\bar{\nu}]$ is $-\n_{\bar{\nu}\zeta}$.

We have $\Omega(\xi,\bar\eta)\zeta=[\n_\xi,\n_{\bar\eta}]\zeta-\n_{[\xi,\bar{\eta}]}\zeta$. Using the result of Proposition~\ref{prop:var_nabla} we get
\begin{equation}
\begin{split}
g\bigl((\delta\Omega)&(\xi,\bar\eta)\zeta,\bar\nu\bigr)\\=&
	g\bigl([(\delta\n)_\xi,\n_{\bar\eta}]\zeta,\bar\nu\bigr)+
	g\bigl([\n_\xi,(\delta\n)_{\bar\eta}]\zeta,\bar\nu\bigr)-
	g((\delta\n)_{[\xi,\bar{\eta}]}\zeta,\bar\nu)\\=&
g\bigl([(\delta\n)_\xi,\n_{\bar\eta}]\zeta,\bar\nu\bigr)+\n_{\n_{\bar{\eta}\xi}}k(\zeta,\bar{\nu})\\=&
g\bigl((\delta\n)_\xi \n_{\bar\eta}\zeta,\bar\nu\bigr)-
g\bigl(\n_{\bar\eta}(\delta\n)_\xi\zeta,\bar\nu\bigr)+
\n_{\n_{\bar{\eta}\xi}}k(\zeta,\bar{\nu})
\\=&
\n_\xi k(\n_{\bar\eta}\zeta,\bar\nu)-\bar\eta\!\cdot\!g\bigl((\delta\n)_\xi\zeta,\bar\nu\bigr)+
g\bigl((\delta\n)_\xi\zeta, \n_{\bar\eta}\bar\nu\bigr)+
\n_{\n_{\bar{\eta}\xi}}k(\zeta,\bar{\nu})\\=&
\n_\xi k(\n_{\bar\eta}\zeta,\bar\nu)-\bar\eta\!\cdot\!\n_\xi k(\zeta,\bar\nu)+\n_\xi k(\zeta,\n_{\bar\eta}\bar\nu)+
\n_{\n_{\bar{\eta}\xi}}k(\zeta,\bar{\nu})\\=&
-\n_{\bar \eta}\n_\xi k(\zeta,\bar\nu)+\n_{\n_{\bar{\eta}\xi}}k(\zeta,\bar{\nu}).
\end{split}
\end{equation}
Here in the second equality we use the facts that $(\delta\n)_{\bar \eta}$ vanishes on (1,0)-vectors and that $(\delta\n)_{[\xi,\bar{\eta}]}\zeta=-(\delta\n)_{\n_{\bar{\eta}\xi}}\zeta$. 
\end{proof}

\section{Evolution of the Chern curvature and torsion under the HCF}\label{sec:evolution}
The main object of our study is the following specification of a general Hermitian curvature flow~\eqref{eq:HCF_intro}
\begin{equation}\label{eq:HCF}
\begin{cases}
\cfrac {d g_{i\bar j}(t)}{d t} = -S^{g(t)}_{i\bar j}-Q^{g(t)}_{i\bar j},\\
g(0)=g_0,
\end{cases}
\end{equation}
where 
\[
S^{g(t)}_{i\bar j}=g^{m\bar n}\Omega_{m\bar n i\bar j}
\quad\mbox{ and }\quad
Q^{g(t)}_{i\bar j}=\frac{1}{2}g^{m\bar n}g^{p\bar s}T_{pm\bar j}T_{\bar s\bar n i}
\] 
are the second Ricci-Chern curvature and a certain quadratic torsion term for $g=g(t)$. By~\cite[Prop.\,5.1]{st-ti-11} there exists unique solution to equation~\eqref{eq:HCF} on some time interval $[0,\tau)$ for some $\tau>0$. Our first goal is to derive evolution equation for the Chern curvature under this HCF. The entire computation is based on Proposition~\ref{prop:var_omega} and uses solely Bianchi identities and commutation of covariant derivatives.

\begin{proposition}\label{prop:omega_evolution}
Assume that $g(t)$ solves the HCF~\eqref{eq:HCF} on $[0;\tau)$. Then the tensor $\Omega(t):=\Omega^{g(t)}$ evolves according to the equation
\begin{equation}\label{eq:dt_Omega}
\begin{split}
\frac{d}{dt}\Omega_{i\bar i j\bar j}&=
	g^{m\bar n}\Bigl(
		\n_{m}\n_{\bar n}\Omega_{i\bar i j \bar j}+
		T_{\bar n\bar i}^{\bar r}\n_m\Omega_{i \bar r j \bar j}+
		T_{m i}^q \n_{\bar n}\Omega_{q \bar i j \bar j}\\&+
		T_{m i}^q T_{\bar n\bar i}^{\bar r} \Omega_{q \bar r j \bar j}+
		g^{p\bar s}(
			T_{pm\bar j}\n_{\bar n}\Omega_{i\bar i j\bar s}+
			T_{\bar s\bar n j}\n_m\Omega_{i\bar i p \bar j}\\&+
			T_{pm\bar j}T^{\bar r}_{\bar n\bar i}\Omega_{i\bar r j\bar s}+
			T_{\bar s\bar n j}T^q_{mi}\Omega_{q \bar i p \bar j}+
			g^{q\bar r}T_{\bar s\bar n j}T_{qm\bar j}\Omega_{i\bar i p \bar r}
		)
	\Bigr)\\&+
	\frac{1}{2}|\n_i T_{\cdot \cdot \bar j}|^2+
	g^{m\bar n}g^{p\bar s}(
		\Omega_{i \bar i m \bar s} \Omega_{p\bar n j\bar j}+
		\Omega_{m \bar i j \bar s} \Omega_{i\bar n p\bar j}-
		\Omega_{m \bar i p \bar j} \Omega_{i\bar s j\bar n}
	)\\&-
	g^{p\bar s}(
		S_{p\bar j}+Q_{p\bar j}
	)\Omega_{i\bar i j\bar s}-
	g^{p\bar s}S_{p\bar i}\Omega_{i\bar sj\bar j}-
	g^{p\bar s}Q_{j\bar s}\Omega_{i\bar i p\bar j}.
\end{split}
\end{equation}
\end{proposition}
\begin{proof}
The coordinate vector fields $e_i=\pd/\pd z_i$ and $e_{\bar i}=\pd/\pd \bar{z}_i$ are holomorphic, so we have $\n_i e_{\bar i}=\n_{\bar i}e_i=0$. Hence 
Proposition~\ref{prop:var_omega} implies that
\[
\frac{\pd }{\pd t}\Omega_{i\bar i j\bar j}=\n_{\bar i}\n_i
\bigl(
	S_{j\bar j}+ Q_{j\bar j}
\bigr)
-g^{p\bar s}(S_{p\bar j}+Q_{p\bar j})\Omega_{i\bar i j\bar s}.
\]
We compute separately $\n_{\bar i}\n_iS_{j \bar j}$ and $\n_{\bar i}\n_i Q_{j\bar j}$.

\medskip
\noindent {\bf Step 1.} Compute $\n_{\bar i}\n_iS_{j \bar j}$.

\noindent Let us first modify term $\n_{\bar i}\n_i S_{j\bar j}$ by applying the second Bianchi identity.
\begin{equation}\label{eq:nn_S_step1}
\begin{split}
\n_{\bar i}\n_i S_{j\bar j}=
	g^{m\bar n}\n_{\bar i}\n_i \Omega_{m \bar n j \bar j}=
	g^{m\bar n}\n_{\bar i}(\n_m\Omega_{i \bar n j\bar j} + T^p_{m i}\Omega_{p \bar n j \bar j})\\=
	g^{m\bar n}(\n_{\bar i}\n_m\Omega_{i \bar n j\bar j}+\n_{\bar i}T^p_{mi}\,\Omega_{p \bar n j \bar j}+T^p_{mi}\n_{\bar i}\Omega_{p \bar n j \bar j}).
\end{split}
\end{equation}
Now we commute $\n_{\bar i}$ with $\n_{m}$.
\begin{equation}
\begin{split}
\n_{\bar i}\n_m&\Omega_{i \bar n j\bar j}=
		\n_m\n_{\bar i}\Omega_{i\bar n j\bar j}+
		\Omega_{m \bar i i}^p \Omega_{p\bar n j\bar j}+
		\Omega_{m \bar i \bar n}^{\bar s} \Omega_{i\bar s j\bar j}+
		\Omega_{m \bar i j}^p \Omega_{i\bar n p\bar j}+
		\Omega_{m \bar i \bar j}^{\bar s} \Omega_{i\bar n j\bar s}
		\\=&
		\n_m\n_{\bar i}\Omega_{i\bar n j\bar j}+g^{p\bar s}
		\bigl(
			\Omega_{m \bar i i \bar s} \Omega_{p\bar n j\bar j}-
			\Omega_{m \bar i p \bar n} \Omega_{i\bar s j\bar j}+
			\Omega_{m \bar i j \bar s} \Omega_{i\bar n p\bar j}-
			\Omega_{m \bar i p \bar j} \Omega_{i\bar n j\bar s}
		\bigr).
\end{split}
\end{equation}
Now we again apply the second Bianchi identity to the first term on the RHS of the latter expression.
\begin{equation}
\begin{split}
\n_m\n_{\bar i}\Omega_{i\bar n j\bar j}=
	\n_{m}(\n_{\bar n}\Omega_{i \bar i j \bar j} + T_{\bar n \bar i}^{\bar s}\Omega_{i \bar s j\bar j})=\\
	\n_{m}\n_{\bar n}\Omega_{i \bar i j \bar j} +
	\n_{m}T_{\bar n \bar i}^{\bar s}\Omega_{i \bar s j\bar j} + 
	T_{\bar n \bar i}^{\bar s}\n_{m}\Omega_{i \bar s j\bar j}.
\end{split}
\end{equation}
Next we use twice the first Bianchi identity
\begin{equation}
\begin{split}
\n_{\bar i} T_{m i}^p&=\Omega_{i \bar i m}^p-\Omega_{m \bar i i}^p=
g^{p\bar s}(\Omega_{i\bar i m\bar s}-\Omega_{m\bar i i \bar s})
,\\
\n_m T_{\bar n \bar i}^{\bar s}&=\Omega_{m \bar n \bar i}^{\bar s}-\Omega_{m \bar i \bar n}^{\bar s}=
g^{p\bar s}(\Omega_{m\bar i p\bar n}-\Omega_{m\bar n p\bar i});
\end{split}
\end{equation}
and once the second Bianchi identity:
\begin{equation}
\n_{\bar i}\Omega_{p \bar n j \bar j}=\n_{\bar n}\Omega_{p \bar i j \bar j}+T_{\bar n \bar i}^{\bar s}\Omega_{p \bar s j \bar j}.
\end{equation}
Collecting everything in~\eqref{eq:nn_S_step1} we get
\begin{equation}\label{eq:nn_S_final}
\begin{split}
\n_{\bar i}\n_i S_{j\bar j}=
	g^{m\bar n}(
		\n_{m}\n_{\bar n}\Omega_{i\bar i j \bar j}+
		T_{\bar n\bar i}^{\bar s}\n_m\Omega_{i \bar s j \bar j}+
		T_{m i}^p \n_{\bar n}\Omega_{p \bar i j \bar j}+
		T_{m i}^p T_{\bar n\bar i}^{\bar s} \Omega_{p \bar s j \bar j}
		)\\+
		g^{m\bar n}g^{p\bar s}(
			\Omega_{i \bar i m \bar s} \Omega_{p\bar n j\bar j}
			+\Omega_{m \bar i j \bar s} \Omega_{i\bar n p\bar j}
			-\Omega_{m \bar i p \bar j} \Omega_{i\bar n j\bar s}			
		)-g^{p\bar s}S_{p\bar i}\Omega_{i\bar sj\bar j}.
\end{split}
\end{equation}

\medskip
\noindent {\bf Step 2.} Compute $\n_{\bar i}\n_iQ_{j \bar j}$.
\begin{equation}\label{eq:nn_Q_step1}
\begin{split}
	2\n_{\bar i}\n_i Q_{j\bar j}&=
	g^{m\bar n}g^{p\bar s}\n_{\bar i}\n_i(T_{p m\bar j} T_{\bar s\bar n j})\\&=
	g^{m\bar n}g^{p\bar s}(
		\n_i T_{pm\bar j}\n_{\bar i}T_{\bar s\bar n j}+
		\n_{\bar i} T_{pm\bar j}\n_iT_{\bar s\bar n j}+
		T_{pm\bar j} \n_{\bar i}\n_i T_{\bar s\bar n j}+
		\n_{\bar i}\n_iT_{pm\bar j} T_{\bar s\bar n j}
	)
\end{split}
\end{equation}
The first two summands in brackets are $|\n_i T_{\cdot \cdot\bar j}|^2$ and $|\n_{\bar i} T_{\cdot \cdot \bar j}|^2$ respectively. We now compute $\n_{\bar i}\n_i T_{\bar s\bar n j}$ and $\n_{\bar i}\n_iT_{p m\bar j}$ using Bianchi identities.
\begin{equation}
\begin{split}
	\n_{\bar i}\n_i T_{\bar s\bar n j}=
		\n_{\bar i}(
			\Omega_{i\bar n j\bar s}-\Omega_{i\bar s j\bar n}
		)=
		\n_{\bar n}\Omega_{i\bar i j\bar s}+T^{\bar r}_{\bar n\bar i} \Omega_{i\bar r j\bar s}-
		\n_{\bar s}\Omega_{i\bar i j\bar n}-T^{\bar r}_{\bar s\bar i} \Omega_{i\bar r j\bar n}.
\end{split}
\end{equation}
Using the fact $T_{pm\bar j}$ is anti-symmetric in $m$ and $p$ we get
\begin{equation}\label{eq:2nd_torsion_derivative_expr1}
\begin{split}
	g^{m\bar n}g^{p\bar s}T_{pm\bar j}\n_{\bar i}\n_i T_{\bar s\bar n j}=&
	g^{m\bar n}g^{p\bar s}T_{pm\bar j}
	(
		\n_{\bar n}\Omega_{i\bar i j\bar s}+T^{\bar r}_{\bar n\bar i} \Omega_{i\bar r j\bar s}-
		\n_{\bar s}\Omega_{i\bar i j\bar n}-T^{\bar r}_{\bar s\bar i} \Omega_{i\bar r j\bar n}
	)\\=&
	2g^{m\bar n}g^{p\bar s}T_{pm\bar j}
	(
		\n_{\bar n}\Omega_{i\bar i j\bar s}+T^{\bar r}_{\bar n\bar i} \Omega_{i\bar r j\bar s}
	).
\end{split}
\end{equation}
To compute $\n_{\bar i}\n_iT_{pm\bar j}$ we start with commuting derivatives.
\begin{equation}
\begin{split}
	\n_{\bar i}\n_iT_{pm\bar j}=&
		\n_i\n_{\bar i}T_{pm\bar j}+
		\Omega_{i\bar i p}^qT_{qm\bar j}+
		\Omega_{i\bar i m}^qT_{pq\bar j}+
		\Omega_{i\bar i \bar j}^{\bar r}T_{pm\bar r}\\=&
		\n_i\n_{\bar i}T_{pm\bar j}+
		g^{q\bar r}(
			\Omega_{i\bar i p \bar r}T_{qm\bar j}-
			\Omega_{i\bar i m\bar r}T_{qp\bar j}-
			\Omega_{i\bar i q \bar j}T_{pm\bar r}
		).
\end{split}
\end{equation}
As in~\eqref{eq:2nd_torsion_derivative_expr1} we rewrite $g^{m\bar n}g^{p\bar s}\n_i\n_{\bar i}T_{pm\bar j}T_{\bar s\bar n j}$ and use the fact $T_{\bar s \bar n j}$ is anti-symmetric in $\bar n$ and $\bar s$. 
\begin{equation}\label{eq:2nd_torsion_derivative_expr2}
\begin{split}
	g^{m\bar n}g^{p\bar s}&\n_{\bar i}\n_iT_{pm\bar j}T_{\bar s\bar n j}\\&=
	g^{m\bar n}g^{p\bar s}T_{\bar s\bar n j}
	\bigl(
	2(
		\n_m\Omega_{i\bar i p \bar j}+T^q_{mi}\Omega_{q \bar i p \bar j}
		)+
	g^{q\bar r}
	(
		\Omega_{i\bar i p \bar r}T_{qm\bar j}-
		\Omega_{i\bar i m\bar r}T_{qp\bar j}-
		\Omega_{i\bar i q \bar j}T_{pm\bar r}
	)
	\bigr)\\&=
	g^{m\bar n}g^{p\bar s}T_{\bar s\bar n j}
	\bigl(
		2(
		\n_m\Omega_{i\bar i p \bar j}+T^q_{mi}\Omega_{q \bar i p \bar j}
		)+
	g^{q\bar r}
	(
		2\Omega_{i\bar i p \bar r}T_{qm\bar j}-
		\Omega_{i\bar i q \bar j}T_{pm\bar r}
	)
	\bigr)\\&=
	2g^{m\bar n}g^{p\bar s}T_{\bar s\bar n j}
		\bigl(
			\n_m\Omega_{i\bar i p \bar j}+T^q_{mi}\Omega_{q \bar i p \bar j}+
		g^{q\bar r}\Omega_{i\bar i p \bar r}T_{qm\bar j}
		\bigr)-
		2g^{q\bar r}\Omega_{i\bar i q \bar j}Q_{j\bar r}.		
\end{split}
\end{equation}
Expressions~\eqref{eq:2nd_torsion_derivative_expr1}, \eqref{eq:2nd_torsion_derivative_expr2}, $|\n_i T_{\cdot\cdot\bar j}|^2$, and $|\n_{\bar i} T_{\cdot\cdot\bar j}|^2$ together give
\begin{equation}\label{eq:nn_Q_final}
\begin{split}
	2\n_{\bar i}\n_iQ_{j\bar j}&=
	|\n_i T_{\cdot \cdot\bar j}|^2+
	|\n_{\bar i} T_{\cdot \cdot \bar j}|^2-
	2g^{q\bar r}\Omega_{i\bar i q\bar j}Q_{j\bar r}+\\&+
	2g^{m\bar n}g^{p\bar s}
	\bigl(
		T_{pm\bar j}\n_{\bar n}\Omega_{i\bar i j\bar s}+
		T_{\bar s\bar n j}\n_m\Omega_{i\bar i p \bar j}\\&+
		T_{pm\bar j}T^{\bar r}_{\bar n\bar i}\Omega_{i\bar r j\bar s}+
		T_{\bar s\bar n j}T^q_{mi}\Omega_{q \bar i p \bar j}+
		g^{q\bar r}T_{\bar s\bar n j}T_{qm\bar j}\Omega_{i\bar i p \bar r}
	\bigr).
\end{split}
\end{equation}

\medskip
\noindent{\bf Step 3.} Collect all terms together.

\noindent Before we sum up terms $\n_{\bar i}\n_i S_{j\bar j}$ and $\n_{\bar i}\n_i Q_{j\bar j}$ let us note that
\begin{equation}
\begin{split}
|\n_{\bar i}T_{\cdot\cdot \bar j}|^2&=g^{m\bar n}g^{p\bar s}\n_{\bar i} T_{pm\bar j}\n_iT_{\bar s\bar nj}=
g^{m\bar n}g^{p\bar s}(\Omega_{m\bar i p\bar j}-\Omega_{p\bar i m\bar j})(\Omega_{i\bar n j\bar s}-\Omega_{i\bar s j\bar n})\\&=
g^{m\bar n}g^{p\bar s}(
	\Omega_{m\bar i p\bar j}\Omega_{i\bar n j\bar s}+
	\Omega_{p\bar i m\bar j}\Omega_{i\bar s j\bar n}-
	\Omega_{m\bar i p\bar j}\Omega_{i\bar s j\bar n}-
	\Omega_{p\bar i m\bar j}\Omega_{i\bar n j\bar s}
	)\\&=
2g^{m\bar n}g^{p\bar s}(
	\Omega_{m\bar i p\bar j}\Omega_{i\bar n j\bar s}-
	\Omega_{m\bar i p\bar j}\Omega_{i\bar s j\bar n}
	).
\end{split}
\end{equation}
Equations~\eqref{eq:nn_S_final},~\eqref{eq:nn_Q_final} with the term $-g^{p\bar s}(S_{p\bar j}+Q_{p\bar j})\Omega_{i\bar i j\bar s}$ give
\begin{equation}\label{eq:dt_Omega_final}
\begin{split}
\frac{d}{dt}\Omega_{i\bar i j\bar j}&=
	\n_{\bar i}\n_i
	\bigl(
		S_{j\bar j}+ Q_{j\bar j}
	\bigr)
	-g^{p\bar s}(S_{p\bar j}+Q_{p\bar j})\Omega_{i\bar i j\bar s}\\&=
	g^{m\bar n}\Bigl(
		\n_{m}\n_{\bar n}\Omega_{i\bar i j \bar j}+
		T_{\bar n\bar i}^{\bar r}\n_m\Omega_{i \bar r j \bar j}+
		T_{m i}^q \n_{\bar n}\Omega_{q \bar i j \bar j}\\&+
		T_{m i}^q T_{\bar n\bar i}^{\bar r} \Omega_{q \bar r j \bar j}+
		g^{p\bar s}(
			T_{pm\bar j}\n_{\bar n}\Omega_{i\bar i j\bar s}+
			T_{\bar s\bar n j}\n_m\Omega_{i\bar i p \bar j}\\&+
			T_{pm\bar j}T^{\bar r}_{\bar n\bar i}\Omega_{i\bar r j\bar s}+
			T_{\bar s\bar n j}T^q_{mi}\Omega_{q \bar i p \bar j}+
			g^{q\bar r}T_{\bar s\bar n j}T_{qm\bar j}\Omega_{i\bar i p \bar r}
		)
	\Bigr)\\&+
	\frac{1}{2}|\n_i T_{\cdot \cdot \bar j}|^2+
	g^{m\bar n}g^{p\bar s}(
		\Omega_{i \bar i m \bar s} \Omega_{p\bar n j\bar j}+
		\Omega_{m \bar i j \bar s} \Omega_{i\bar n p\bar j}-
		\Omega_{m \bar i p \bar j} \Omega_{i\bar s j\bar n}
	)\\&-
	g^{p\bar s}(
		S_{p\bar j}+Q_{p\bar j}
	)\Omega_{i\bar i j\bar s}-
	g^{p\bar s}S_{p\bar i}\Omega_{i\bar sj\bar j}-
	g^{p\bar s}Q_{j\bar s}\Omega_{i\bar i p\bar j}.
\end{split}
\end{equation}
\end{proof}

\begin{definition}
Let $u_{i\bar i j\bar j}$ be a curvature-type tensor field. Tensor $w=w_{i\bar i j\bar j}$ is a \emph{first order variation of $u$} if it is of the form
\[
w_{i\bar ij\bar j}=
A^{p}_{i}u_{p\bar ij\bar j}+
B^{\bar s}_{\bar i}u_{i \bar sj\bar j}+
C^{p}_{j}u_{i\bar ip\bar j}+
D^{\bar s}_{\bar j}u_{i \bar ij\bar s}+a^p\n_p u_{i\bar i j\bar j}+
b^{\bar s}\n_{\bar s} u_{i\bar i j\bar j}
\]
for some complex $n\times n$ matrices $A,B,C,D$, and vectors $a,b$. We will write $F_1(u)$ to denote (any) first order variation of $u$.
\end{definition}

The reason why we introduce this notion will become apparent in the next section. Roughly speaking, the first order variations of the Chern curvature $\Omega$ entering the evolution equation~\eqref{eq:dt_Omega} ``do not affect'' the preservation of positivity of $\Omega$ along the HCF. Below we are interested in the evolution equation for $\Omega$ only up to the first order variation of $\Omega$.

\begin{example}
Terms
$g^{p\bar s}(S_{p\bar j}+\frac{1}{2}Q_{p\bar j})\Omega_{i\bar i j\bar s}$, 
$g^{p\bar s}S_{p\bar i}\Omega_{i\bar sj\bar j}$ and $\frac{1}{2}g^{p\bar s}Q_{j\bar s}\Omega_{i\bar i p\bar j}$ in~\eqref{eq:dt_Omega}
are first-order variations of~$\Omega$.
\end{example}

\begin{proposition}
The difference between $g^{m\bar n}\n_m\n_{\bar n}\Omega_{i\bar i j\bar j}$ and the real Chern Laplacian $\Delta=\frac{1}{2}(\n_m \n_{\bar n} + \n_{\bar n}\n_m) \Omega_{i\bar i j\bar j}$ is a first order variation of $\Omega$.
\end{proposition}
\begin{proof}
The difference between operators $g^{m\bar n}\n_m\n_{\bar n}$ and $\frac{1}{2}(\n_m \n_{\bar n} + \n_{\bar n}\n_m)$ is the zero-order operator, given by the commutator $\frac{1}{2}g^{m\bar n}[\n_m,\n_{\bar n}]$. Clearly, the result of the action of $\frac{1}{2}g^{m\bar n}[\n_m,\n_{\bar n}]$ on $\Omega$
\[
-\frac{1}{2}g^{m\bar n}(
		\Omega_{m\bar n i}^p\Omega_{p\bar i j\bar j}+
		\Omega_{m\bar n \bar i}^{\bar s}\Omega_{i \bar s j\bar j}+
		\Omega_{m\bar n j}^p\Omega_{i\bar i p\bar j}+
		\Omega_{m\bar n \bar j}^{\bar s}\Omega_{i \bar i j\bar s}
	)
\]
is a first-order variation of $\Omega$.
\end{proof}

Using this proposition we can abbreviate equation~\eqref{eq:dt_Omega} as follows

\begin{equation}\label{eq:dt_Omega_mod_1st_order}
\begin{split}
\frac{d}{dt}\Omega_{i\bar i j\bar j}&=
	\Delta \Omega_{i\bar i j \bar j} + g^{m\bar n}\Bigl(
		T_{\bar n\bar i}^{\bar r}\n_m\Omega_{i \bar r j \bar j}+
		T_{m i}^q \n_{\bar n}\Omega_{q \bar i j \bar j}\\&+
		T_{m i}^q T_{\bar n\bar i}^{\bar r} \Omega_{q \bar r j \bar j}+
		g^{p\bar s}(
			T_{pm\bar j}\n_{\bar n}\Omega_{i\bar i j\bar s}+
			T_{\bar s\bar n j}\n_m\Omega_{i\bar i p \bar j}\\&+
			T_{pm\bar j}T^{\bar r}_{\bar n\bar i}\Omega_{i\bar r j\bar s}+
			T_{\bar s\bar n j}T^q_{mi}\Omega_{q \bar i p \bar j}+
			g^{q\bar r}T_{\bar s\bar n j}T_{qm\bar j}\Omega_{i\bar i p \bar r}
		)
	\Bigr)\\&+
	\frac{1}{2}|\n_i T_{\cdot \cdot \bar j}|^2+
	g^{m\bar n}g^{p\bar s}(
		\Omega_{i \bar i m \bar s} \Omega_{p\bar n j\bar j}+
		\Omega_{m \bar i j \bar s} \Omega_{i\bar n p\bar j}-
		\Omega_{m \bar i p \bar j} \Omega_{i\bar s j\bar n}
	)\\&+
	F_1(\Omega)_{i\bar i j\bar j},
\end{split}
\end{equation}
where $F_1(\Omega)$ is a first order variation of $\Omega$.

\subsection{Torsion-twisted connection}

The evolution equation of the Chern curvature in the form \eqref{eq:dt_Omega_mod_1st_order} involves many terms and it is difficult to analyze it directly. To simplify the equation and underline some of its ``positivity'' properties we introduce \emph{torsion-twisted} connection $\n^T$ on the space of curvature tensors and use its Laplacian to rewrite the evolution equation. Specifically, the 8 terms in the brackets in~\eqref{eq:dt_Omega_mod_1st_order} will be `absorbed' by the Laplacian of a \emph{torsion-twisted} connection on the space of curvature tensors.

\begin{definition}[Torsion-twisted connections]\label{def:torsion_twisted_connection}
We define $\n^1, \n^2$ to be two \emph{torsion-twisted} connections on $TM$ given by the following identities
\begin{equation}\label{eq:torsion_twisted_connection}
\begin{split}
&{\n^1}_X Y=\n_X Y - T(X,Y),\\
&\n^2_X Y=\n_X Y+g(Y,T(X,\cdot))^\#,
\end{split}
\end{equation}
where $X, Y, Z$ are sections of $T M$, $\n$ is the Chern connection and $\#\colon T^*M\to TM$ is the isomorphism induced by $g$.	
Equivalently, in the coordinates for a vector $\xi=\xi^p \frac{\pd}{\pd z^p}$ one has
\begin{equation}
\begin{split}
&\n^1_i \xi^p=\n_i\xi^p - T^p_{ij}\xi^j, \quad
\n^1_{\bar i}\xi^p=\n_{\bar i}\xi^p\\
&\n^2_i \xi^p=\n_i\xi^p, \quad
\n^2_{\bar i}\xi^p=\n_{\bar i}\xi^p + g^{p\bar s}T_{\bar i\bar s j}\xi^j.
\end{split}
\end{equation}
%As usual, we extend $\n^1$ and $\n^2$ to connections on $T^{0,1}M$ via conjugations and to all vector bundles associated with $T^{1,0}M$ via Leibniz rule.
\end{definition}

\begin{remark}
It is easy to check that $g$ considered as a section of $T^*M\otimes T^*M$ is parallel with respect to the connection $\n_1\otimes\n_2$, i.e.,
for any vector fields $X,Y,Z\in \Gamma(T M)$ we have
\[
X\!\cdot\!g(Y,Z)=g(\n^1_X Y, Z) + g(Y,\n^2_{X} Z).
\]
In other words, $\n^2$ is \emph{dual conjugate} to ${\n^1}$ via $g$.
\end{remark}

\begin{definition}[Torsion-twisted connection on the space of curvature tensors]
Now let us define the torsion-twisted connection $\n^T$ on the space of curvature tensors $\Lambda^{1,0}M\otimes\Lambda^{0,1}M\otimes\Lambda^{1,0}M\otimes\Lambda^{0,1}M$.
\[
\n^T:=\n^1\otimes\bar{\n^1}\otimes\n^2\otimes\bar{\n^2}.
\]
\end{definition}

The use of the torsion-twisted connection allows to simplify significantly the evolution equation~\eqref{eq:dt_Omega_mod_1st_order}.

\begin{lemma}\label{lm:torsion_twisted_laplacian}
Let $\Delta^T=\frac{1}{2}g^{m\bar n}(\n_m^T\n_{\bar n}^T+\n_{\bar n}^T\n_m^T)$ be the Laplacian of the torsion-twisted connection and let $u$ be a curvature-type tensor. Then the following identity holds.
\begin{equation}
\begin{split}
\Delta^T u_{i\bar i j\bar j}&=
	\Delta u_{i\bar i j \bar j} + g^{m\bar n}\Bigl(
		T_{\bar n\bar i}^{\bar r}\n_m u_{i \bar r j \bar j}+
		T_{m i}^q \n_{\bar n} u_{q \bar i j \bar j}\\&+
		T_{m i}^q T_{\bar n\bar i}^{\bar r} u_{q \bar r j \bar j}+
	g^{p\bar s}(
				T_{pm\bar j}\n_{\bar n}u_{i\bar i j\bar s}+
				T_{\bar s\bar n j}\n_m u_{i\bar i p \bar j}\\&+
				T_{pm\bar j}T^{\bar r}_{\bar n\bar i}u_{i\bar r j\bar s}+
				T_{\bar s\bar n j}T^q_{mi}u_{q \bar i p \bar j}+
				g^{q\bar r}T_{\bar s\bar n j}T_{qm\bar j}u_{i\bar i p \bar r}
			)
	\Bigr)\\&+
	F_1(u)_{i\bar i j\bar j},
\end{split}
\end{equation}
where $F_1(u)$ is a first order variation of $u$.
\end{lemma}
\begin{proof}
By the definition of $\n^T$ we have
\begin{equation}
\begin{split}
\n^T_m u_{i \bar i j\bar j}=
	\n_m u_{i \bar i j\bar j}+
	T_{mi}^q u_{q \bar i j\bar j}-
	g^{q\bar r}T_{mq\bar j}u_{i \bar i j\bar r},\\
\n^T_{\bar n}u_{i \bar i j\bar j}=
	\n_{\bar n} u_{i \bar i j\bar j}+
	T_{\bar n\bar i}^{\bar r} u_{i \bar r j\bar j}-
	g^{q\bar r} T_{\bar n\bar r j}u_{i \bar i q\bar j}.
\end{split}
\end{equation}
Next we compute
\begin{equation}
\begin{split}
\n^T_{\bar n}\n^T_m u_{i \bar i j\bar j}=&
	\n^T_{\bar n}\bigl(
		\n_m u_{i \bar i j\bar j}+
		T_{m i}^q u_{q \bar i j\bar j}-
		g^{q\bar r}T_{mq\bar j}u_{i \bar i j\bar r}
	\bigr)\\=&
	\n_{\bar n} \n_m u_{i \bar i j\bar j}+
		T_{\bar n\bar i}^{\bar r} \n_m u_{i \bar r j\bar j}-
		g^{q\bar r} T_{\bar n\bar r j}\n_m u_{i \bar i q\bar j}\\&+
		\n^T_{\bar n}(T_{m i}^q u_{q \bar i j\bar j})-
		\n^T_{\bar n}(g^{q\bar r}T_{mq\bar j}u_{i \bar i j\bar r})\\=&
	\n_{\bar n} \n_m u_{i \bar i j\bar j}+
	T_{\bar n\bar i}^{\bar r} \n_m u_{i \bar r j\bar j}-
	g^{q\bar r} T_{\bar n\bar r j}\n_m u_{i \bar i q\bar j}\\&+
	T_{mi}^q \n_{\bar n}u_{q\bar ij\bar j}+
	\n_{\bar n} T_{mi}^q u_{q\bar i j\bar j}+
	T_{mi}^qT_{\bar n\bar i}^{\bar r} u_{q\bar rj\bar j}-
	g^{p\bar s}T_{\bar n\bar s j} T_{m i}^q u_{q\bar i p\bar j}\\&-
	g^{q\bar r}\bigl(
		T_{mq\bar j}\n_{\bar n}u_{i\bar i j\bar r} +
		\n_{\bar n}T_{mq\bar j} u_{i\bar i j\bar r} +
		T_{\bar n\bar i}^{\bar s}T_{mq\bar j}u_{i\bar s j\bar r} -
		g^{p\bar s}T_{\bar n\bar s j} T_{mq\bar j} u_{i\bar i p\bar r}
	\bigr).
\end{split}
\end{equation}
Computation of $\n^T_m\n_{\bar n}^Tu_{i\bar ij\bar j}$ is analogous. It remains to use the fact that the torsion $T$ is antisymmetric in its two first arguments and notice that $\n_{\bar n}T_{mq\bar j} u_{i\bar i j\bar r}, \n_{\bar n} T_{mi}^q u_{q\bar i j\bar j}$ are first-order variations of $u$. After contraction with $g^{m\bar n}$ we get the desired formula.
\end{proof}

With the use of this lemma we can rewrite the evolution equation~\eqref{eq:dt_Omega} as follows.

\begin{proposition}\label{prop:dt_Omega_torsion_twisted}
Assume that $g(t)$ solves the HCF~\eqref{eq:HCF} on $[0;\tau)$. Then the tensor $\Omega_{i\bar i j\bar j}(t)$ evolves according to the equation
\begin{equation}\label{eq:dt_Omega_torsion_twisted}
\begin{split}
\frac{d}{dt}\Omega_{i\bar i j\bar j}&=
	\Delta^T\Omega_{i\bar i j\bar j}+
	Q_2(\n T)_{i\bar i j\bar j}+F_2(\Omega)_{i\bar i j\bar j}+F_1(\Omega)_{i\bar i j\bar j},
\end{split}
\end{equation}
where $F_1(\Omega)$ is a first order variation of $\Omega$ and $F_2(\Omega)$ is the quadratic curvature term given by
\begin{equation}\label{eq:quadratic_curvature}
F_2(\Omega)_{i\bar i j\bar j}=g^{m\bar n}g^{p\bar s}(
		\Omega_{i \bar i m \bar s} \Omega_{p\bar n j\bar j}+
		\Omega_{m \bar i j \bar s} \Omega_{i\bar n p\bar j}-
		\Omega_{m \bar i p \bar j} \Omega_{i\bar s j\bar n}
	)
\end{equation}
and $Q_2(\n T)$ is the partial square norm of the covariant derivative of the torsion.
\begin{equation}\label{eq:Q_2_nT}
Q_2(\n T)_{i\bar i j\bar j}=
\frac{1}{2}|\n_i T_{\cdot \cdot \bar j}|^2=
\frac{1}{2}g^{p\bar s}g^{m\bar n}\n_iT_{pm\bar j}\n_{\bar i}T_{\bar s\bar n j}.
\end{equation}
\end{proposition}

\begin{remark}
The term $F_2(\Omega)$ coincides with the quadratic term, computed by Hamilton~\cite{ha-86} in the case of Ricci flow. If the underlying manifold $M$ is K\"ahler then the torsion $T$ is zero, so $\Delta^T=\Delta$, and the equation~\eqref{eq:dt_Omega_torsion_twisted} recovers the original evolution of the curvature under the Ricci flow~\cite{ha-86,mo-88}.
\end{remark}

\section{Properties of non-negative curvature-type tensors}\label{sec:positivity_properties}
Throughout this section $u$ is a curvature-type tensor field on a Hermitian complex manifold $(M, g, J)$. Moreover, we assume that $u$ is non-negative (see Definition~\ref{def:positive}). To distinguish a tensor field $u$ from its value at a particular point $x\in M$, we will put the corresponding subscript $u_x:=u|_{T_xM}$.
In this section we study the structure of the zero locus of $u$. We start with the following elementary observation.
\begin{proposition}\label{prop:vanish}
Let $u$ be a tensor field of curvature type. Assume that $u\ge 0$ and $u_x(\xi, \bar \xi, \eta, \bar \eta)=0$ for $\xi,\eta\in T_x^{1,0}M$. Then for any $\zeta\in T_x^{1,0}M$ tensor $u$ satisfies
\begin{equation}
\begin{split}
&u_x(\xi, \bar \zeta, \eta, \bar \eta)=u_x(\zeta, \bar \xi, \eta, \bar \eta)=0;\\
&u_x(\xi, \bar \xi, \eta, \bar \zeta)=u_x(\xi, \bar \xi, \zeta, \bar \eta)=0;\\ 
&(\n u)_x(\xi, \bar \xi, \eta, \bar \eta)=0.
\end{split}
\end{equation}
In particular, for any first order variation of $u$ 
\[
F_1(u)(\xi,\bar{\xi},\eta,\bar{\eta})=0.
\]
\end{proposition}
\begin{proof}
Let us fix $\zeta\in T_x^{1,0}M$. Consider the first variation of the expression $u_x(\xi, \bar \xi, \eta, \bar \eta)$ along the fibre, specifically consider the function $\Phi\colon \R\to \R$ given by the formula
\[
\Phi(s)=u_x(\xi+s\zeta, \bar{\xi + s\zeta}, \eta, \bar \eta).
\]
Since $u$ is Griffiths non-negative, $\Phi(s)\ge 0$. Hence $\Phi$ attains its minimum at~$0$, so
\[
0=\frac{d\Phi(s)}{ds}\bigg|_{s=0}=
u_x(\zeta,\bar \xi, \eta, \bar \eta)+u_x(\xi,\bar{ \zeta}, \eta, \bar \eta)=
2{\rm Re}\bigl(u_x(\xi,\bar \zeta, \eta, \bar \eta)\bigr).
\]
Since $\zeta$ is arbitrary, it implies that $u_x(\xi,\bar \zeta, \eta, \bar \eta)=0$ for any $\zeta\in T_x^{1,0}M$.

The vanishing of the second term is proved similarly by considering the variation $\Phi(s)=u_x(\xi, \bar \xi, \eta+s\zeta, \bar{\eta + s \zeta}).$ The third identity follows from the vanishing of the variation of $u_x(\xi,\bar{\xi},\eta,\bar{\eta})$ in $x\in M$.
\end{proof}
We have used the first order variation of $u_x(\xi,\bar\xi,\eta,\bar\eta)$ in its vector arguments to prove the vanishing results in Proposition~\ref{prop:vanish}. The next proposition uses the second variation in the vector arguments of $u_x(\xi,\bar\xi,\eta,\bar\eta)$ to prove a more delicate inequality. 
\begin{proposition}\label{prop:var_fiber}
Let $u_x$ be a tensor of curvature type. Let $e_1,\dots,e_n\in T_x^{1,0}M$ be an orthonormal basis. Then for $\xi,\eta\in T_x^{1,0}M$ there exists a positive constant $K=K(\xi,\eta)$, such that tensor $u_x$ satisfies
\begin{equation}\label{eq:quadratic_ineq}
\begin{split}
\sum_{i,j}
\bigl(
	u_x(\xi,\bar\xi,e_i,\bar e_j)u_x(e_j,\bar e_i,\eta,\bar\eta)-
	u_x(\xi,\bar e_i,\eta, \bar e_j)u_x(e_j,\bar\xi,e_i, \bar\eta)
\bigr)\\
\ge K\cdot\!\!\!\! \inf_{\substack{\nu,\zeta\in T_x^{1,0}M\\ |\nu|^2+|\zeta|^2\le 1}} {\delta^2} u_x(\nu,\zeta),
\end{split}
\end{equation}
where
\[
{\delta^2} u_x(\nu,\zeta):=\frac{d^2}{ds^2}u_x(\xi+s\nu, \bar \xi+s\bar\nu, \eta+s\zeta, \bar \eta+s\bar\zeta)|_{s=0}
\]
is the second variation of $u_x$ in its vector arguments.
\end{proposition}

\begin{proof}
First let us denote the infimum on the RHS of~\eqref{eq:quadratic_ineq} by $-K_0$. Since ${\delta^2} u_x(0,0)=0$, we have $K_0\ge 0$. Now we fix vectors $\chi,\psi\in T_x^{1,0}M$ and consider ${\delta^2} u_x(\chi,\psi)$. Clearly
\[
{\delta^2} u_x(\chi,\psi)\ge -K_0(g(\chi,\bar \chi)+g(\psi,\bar\psi)). 
\]
Hence the Hermitian form 
\[
Q_u(\chi\oplus\psi,\bar{\chi\oplus\psi}):=
\frac{1}{2}\bigl(
{\delta^2} u_x(\chi,\psi)+
{\delta^2} u_x(i\chi,-i\psi)
\bigr)+
K_0(g(\chi,\bar \chi)+g(\psi,\bar\psi))
\]
is positive semidefinite. Let us denote the matrices of quadratic forms $u_x(\xi, \bar \xi, \cdot, \bar\cdot)$, $u_x(\xi, \bar \cdot, \eta, \bar \cdot)$, $u_x(\cdot, \bar \cdot, \eta, \bar \eta)$ in the unitary orthonormal basis by ${\bf A}_0, {\bf B}_0, {\bf C}_0$ respectively. Then $Q_u$ satisfies
\begin{equation}\label{eq:Q_2nd_var}
Q_u(\chi\oplus \psi,\bar {\chi\oplus \psi})=
({\bf A}_0+K_0{\bf I}_n)(\psi,\bar \psi)+
2{\rm Re}{\bf B}_0(\chi,\psi)+
({\bf C}_0+K_0{\bf I}_n)(\chi,\bar \chi).
\end{equation}

We will need the following Lemma.
\begin{lemma}\label{lm:matrix_ineq}
Let ${\bf A}=(a_{ij}),{\bf B}=(b_{ij}),{\bf C}=(c_{ij})$ be real $n\times n$ matrices with ${\bf A}$ and $\bf C$~--- symmetric. Suppose that the real quadratic form
\[
Q(v\oplus w,{v\oplus w})={\bf A}(v, v)+2{\bf B}(v,w)+{\bf C}(w, w)
\]
on $\R^{2n}$ is positive semidefinite. Then we have
\[
{\rm tr} ({\bf AC})=\sum_{i,j}a_{ij}c_{ji}\ge
\sum_{i,j}b_{i j} b_{j i}={\rm tr}({\bf B}^2).
\]
\end{lemma}

\begin{proof}[Proof of the lemma]
Since $Q$ is a positive semidefinite quadratic form, its matrix in the standard basis $E_1,\dots,E_{2n}$ of $\R^{2n}$ is non-negative.
\[
{\bf M}_1=
\left[
\begin{matrix}
\bf A & \bf B \\
{\bf B}^T & \bf C
\end{matrix}
\right]\ge 0.
\]
Similarly the matrix of $Q$ in basis $E_{n+1},\dots, E_{2n},-E_1,\dots-E_n$ is non-negative.
\[
{\bf M}_2=
\left[
\begin{matrix}
\bf C & -{\bf B}^T \\
-{\bf B} & \bf A
\end{matrix}
\right]\ge 0.
\]
Trace of the product of two positive semidefinite matrices is non-negative, hence
\[
\tr({\bf M}_1{\bf M}_2)=2\sum_{i,j}(a_{ij}c_{ji}-b_{ij}b_{ji})\ge 0,
\]
which implies the required inequality. The lemma is proved.
\end{proof}
Now we consider a real positive semidefinite quadratic form on $\C^{n}\oplus\C^n\simeq \R^{4n}$ 
\[
Q(v\oplus w,\bar{v\oplus w})={\bf A}(v, \bar v)+2{\rm Re}{\bf B}(v,w)+{\bf C}(w, \bar w),
\]
where ${\bf A}=(a_{ij}),{\bf B}=(b_{ij}),{\bf C}=(c_{ij})$ are complex $n\times n$ matrices with $\bf A$, $\bf C$~--- Hermitian, i.e, $a_{ij}=\bar{a_{ji}}$, $c_{ij}=\bar{c_{ji}}$. Applying Lemma~\ref{lm:matrix_ineq} in this situation we get
\begin{equation}
\sum_{i,j} a_{ij}\bar{c_{ij}}\ge \sum_{i,j} b_{ij}\bar{b_{ji}}.
\end{equation}
In particular, for the form $Q_u$ in~\eqref{eq:Q_2nd_var} one gets
\[
{\rm tr}\bigl(({\bf A}_0+K_0{\bf I}_n)\bar{({\bf C}_0+K_0{\bf I}_n)}\bigr)\ge 
\tr ({\bf B}_0\bar{{\bf B}_0}),
\]
\[
{\rm tr}({\bf A}_0\bar{{\bf C}_0})-\tr ({\bf B}_0\bar{{\bf B}_0})
\ge -K_0({\rm tr} {\bf A}_0+{\rm tr} {\bf B}_0+nK_0).
\]
Setting $K={\rm tr} {\bf A}_0+{\rm tr} {\bf B}_0+nK_0$ and plugging the definitions of ${\bf A}_0,{\bf B}_0,{\bf C}_0$, we get the required inequality~\eqref{eq:quadratic_ineq}.
\end{proof}

Lemma~\ref{lm:matrix_ineq} was used by Cao~\cite[Lemma\,4.2]{ca-92} to deduce Harnack's inequality along the K\"ahler-Ricci flow. In the Hermitian setting one has to be careful, since $u=\Omega_{i\bar j k\bar l}$ is not symmetric in its first and third arguments, so the corresponding matrix $\bf B$ is not symmetric.

\begin{corollary}\label{cor:var_fiber}
Let $u_x$ be a tensor of curvature type. Assume that $u_x\ge 0$ and $u_x(\xi, \bar \xi, \eta, \bar \eta)=0$ for $\xi,\eta\in T_x^{1,0}M$. Let $e_1,\dots,e_n\in T_x^{1,0}M$ be an orthonormal basis. Then tensor $u_x$ satisfies
\begin{equation}
\sum_{i,j}
\bigl(
	u_x(\xi,\bar\xi,e_i,\bar e_j)u_x(e_j,\bar e_i,\eta,\bar\eta)-
	u_x(\xi,\bar e_i,\eta, \bar e_j)u_x(e_j,\bar\xi,e_i, \bar\eta)
\bigr)
\ge 0.
\end{equation}
\end{corollary}
\begin{proof}
Since $u_x\ge 0$ and $u_x(\xi, \bar \xi, \eta, \bar \eta)=0$, the second variation $\delta^2 u_x$ in any direction is non-negative, hence the RHS in~\eqref{eq:quadratic_ineq} is zero.
\end{proof}

In the K\"ahler case this corollary was proved by Bando in~\cite[Prop.\,1]{ba-84} for $n=3$ and for an arbitrary~$n$ in~\cite[Prop.\,1.1]{mo-88} by Mok. Proposition~\ref{prop:var_fiber} in the K\"ahler situation was proved in~\cite{gu-09}. Our proof of this proposition follows the argument of Cao~\cite{ca-92}, see also~\cite[\S 5.2.3]{krf-13}.

\section{Griffiths positivity under the HCF}\label{sec:main}
\subsection{Preservation of Griffiths positivity}
With the results of Section~\ref{sec:positivity_properties} we are ready to prove our main result. 

\begin{theorem}\label{thm:main}
Let $g(t), t\in[0,\tau)$ be the solution to the HCF~\eqref{eq:HCF} on a compact complex Hermitian manifold $(M,g_0,J)$. Assume that the Chern curvature $\Omega^{g_0}$ at the initial moment $t=0$ is Griffiths non-negative (resp. positive), i.e., for $\xi,\eta\in T^{1,0}M$
\[
\Omega^{g_0}(\xi,\bar\xi,\eta,\bar\eta)\ge 0\quad (\mbox{resp.}>0).
\]
Then for $t\in[0,\tau)$ the Chern curvature $\Omega(t)=\Omega^{g(t)}$ remains Griffiths non-negative (resp. positive).
\end{theorem}

To prove Theorem~\ref{thm:main} we adopt the ``barrier'' argument of~\cite[Prop.\,1]{ba-84} to our case. Essentially this result follows from Hamilton's maximum principle for tensors~\cite[Th.\,4.3]{ha-86}.

\begin{proof}[Proof of Theorem~\ref{thm:main}]
We prove only the ``non-negative'' statement. The proof of the ``positive'' statement is similar.
 
Clearly the set $\{t\,|\,\Omega(t)\ge 0\}$ is closed. We are going to prove that this set is open, so further we shrink the initial time interval $[0;\tau)$, if necessary, without explicitly specifying it. In particular, from the very beginning we assume that each of the quantities $|T|,|\n T|,|\Omega|$ is uniformly bounded on $[0;\tau)$. Let us introduce the time-dependent tensor field $\Omega'_{i\bar j k\bar l}:=g_{i\bar j}g_{k\bar l}$. 
Consider
\[
u^\epsilon=\Omega+\epsilon \Omega',
\]
where $\epsilon\colon [0;\tau)\to (0;+\infty)$ is a function to be defined later. Our first aim is to derive differential \emph{inequality} of parabolic type for $u^{\epsilon}$. Rewriting evolution equation~\eqref{eq:dt_Omega_torsion_twisted}, we have
\begin{multline}\label{eq:parabolic_eq}
\frac{d}{dt} u^\epsilon=
	\Delta^T u^\epsilon+F_1(u^\epsilon) + F_2(u^\epsilon) + \frac{d\epsilon}{dt}\Omega'+\epsilon\frac{d}{dt}\Omega'\\ - 
	\epsilon \Delta^T(\Omega')+
	(F_1(\Omega)-F_1(\Omega+\epsilon \Omega'))+
	(F_2(\Omega)-F_2(\Omega+\epsilon \Omega'))+
	Q_2(\n T).
\end{multline}
Let us bound summands in~\eqref{eq:parabolic_eq} separately.
\begin{itemize}
\item clearly $Q_2(\n T)\ge 0$, see~\eqref{eq:Q_2_nT};

\item since $F_1+F_2$ is a smooth map of curvature-type tensors and as $\Omega'>0$, there exist a constant $C_0>0$ such that on $M\times[0;\tau)$ one has
\[
(F_1+F_2)(\Omega)-(F_1+F_2)(\Omega+s\Omega')\ge 
-|s|C_0\Omega', \mbox{ for } s\in \R, |s|\le 1;
\]
\item there is a constant $C_1$ such that on $M\times[0;\tau)$ the torsion-twisted Laplacian and time derivative of $\Omega'$ satisfy
\[
|\Delta^T\Omega'|<C_1\Omega',\quad
\bigl|\frac{d}{dt}\Omega'\bigr|<C_1\Omega'.
\]
\end{itemize}
Assuming that $0<\epsilon<1$ on $[0;\tau)$, from equation~\eqref{eq:parabolic_eq} and the bounds on its summands, we have
\begin{equation}
\frac{d}{dt} u^\epsilon\ge
	\Delta^T u^\epsilon+F_1(u^\epsilon) +F_2( u^\epsilon) + 
	\Bigl(
		\frac{d\epsilon}{dt}-
		(C_0+2C_1)\epsilon
	\Bigr)\Omega'
\end{equation}
We set $\epsilon(t)=\epsilon_0 e^{Kt}$, where $\epsilon_0$ is a small positive number and $K=(C_0+2C_1+1)$. For this $\epsilon$, tensor $u^\epsilon$ satisfies differential inequality (on a smaller time interval such that $\epsilon_0e^{K\tau}< 1$)
\[
\frac{d}{dt} u^\epsilon\ge
	\Delta^T u^\epsilon+F_1(u^\epsilon)+F_2( u^\epsilon) +
	{\epsilon_0}\Omega'.
\]

We claim that $u^{\epsilon}>0$ on $[0;\tau)$, where from now on $\tau$ is fixed and independent of $\epsilon_0$. Clearly $u^{\epsilon}>0$ for $t=0$, since $\epsilon(0)>0$, $\Omega'>0$ and $\Omega|_{t=0}\ge0$. We prove by contradiction and let $t_0>0$ be the first time such that $u^{\epsilon}$ is not strictly positive. Then for some $x\in M$ and non-zero $\xi,\eta\in T_x^{1,0}M$ we have
\[
u_x^\epsilon(\xi,\bar\xi,\eta,\bar\eta)=0.
\]
For these $x$, $t_0$, $\xi$ and $\eta$ we have:
\begin{enumerate}
\item  $(\Delta^T u^\epsilon)_x(\xi,\bar\xi,\eta,\bar\eta)\ge 0$, since $u^\epsilon$ attains local minimum at $(x,t_0)\in M\times (0,\tau)$, $\xi,\eta\in T_x^{1,0}M$;
\item $F_1(u^\epsilon)_x(\xi,\bar\xi,\eta,\bar\eta)=0$ by Proposition~\ref{prop:vanish};
\item $F_2(u^\epsilon)_x(\xi,\bar\xi,\eta,\bar\eta)\ge 0$ by Corollary~\ref{cor:var_fiber};
\item $\Omega'_x(\xi,\bar{\xi},\eta,\bar{\eta})=g(\xi,\bar{\xi})g(\eta,\bar{\eta})>0$.
\end{enumerate}
Hence at $t_0$ one has
\begin{equation}
\begin{split}
0&\ge \frac{d}{dt}
	\bigl(
		u_x^\epsilon(\xi,\bar\xi,\eta,\bar\eta)
	\bigr)=
\bigl(
	\frac{d}{dt} u^\epsilon_x
\bigr)
(\xi,\bar\xi,\eta,\bar\eta)\\
&\ge
(\Delta^T u^\epsilon)_x(\xi,\bar\xi,\eta,\bar\eta)+
F_1(u^\epsilon)_x(\xi,\bar\xi,\eta,\bar\eta)+
F_2( u^\epsilon)_x(\xi,\bar\xi,\eta,\bar\eta)+
\epsilon_0\Omega'_x(\xi,\bar\xi,\eta,\bar\eta)>
0.
\end{split}
\end{equation}
This is a contradiction. Therefore we get $u^\epsilon=\Omega+\epsilon \Omega'>0$ for all sufficiently small $\epsilon_0>0$ and any $t\in[0;\tau)$. Letting $\epsilon_0\to 0$ we conclude $\Omega\ge 0$ for $t\in[0;\tau)$.

So the closed set $\{t|\Omega(t)\ge 0\}$ is also open, therefore $\Omega(t)\ge 0$ for any $t\in[0;\tau)$, provided $\Omega(0)\ge 0$ and the HCF~\eqref{eq:HCF} exist on $[0;\tau)$. This proves the theorem.
\end{proof}

\subsection{Strong maximum principle}

Theorem~\ref{thm:main} states a weak maximum principle for Griffiths positivity along Hermitian curvature flow~\eqref{eq:HCF}. In this section we prove an extension of Brendle and Schoen's strong maximum principle. It was first proved in \cite{br-sc-08} for the isotropic curvature under the Ricci flow. Gu~\cite{gu-09} used this maximum principle for K\"ahler manifolds with non-negative holomorphic bisectional curvature to give a simple proof of the uniformization theorem of Mok~\cite{mo-88}. In~\cite[App.A]{wi-13} Wilking generalized this strong maximum principle for any Lie algebraic Ricci flow invariant condition. We adopt the argument of Wilking to prove the strong maximum principle for Griffiths positivity along the HCF.

Let us use connection $\n^1\oplus\n^2$ (see Definition~\ref{def:torsion_twisted_connection}) to define the parallel transport on the vector bundle $T^{1,0}M\oplus T^{1,0}M$. Namely, the parallel lift of a curve $\gamma\colon (-\epsilon,\epsilon)\to M$ is the solution to the following ODE for $\xi$ and $\eta$.
\begin{equation}\label{eq:parallel_transport}
\begin{cases}
\n_{\gamma'}\xi=T(\gamma',\xi), \\
\n_{\gamma'}\eta=-g(\eta,T(\gamma',\cdot))^\#.
\end{cases}
\end{equation}

\begin{theorem}\label{thm:brendle_schoen}
Let $g(t), t\in[0,\tau)$ be a solution to the HCF~\eqref{eq:HCF} on a compact complex Hermitian manifold $(M,g_0,J)$. Assume that the Chern curvature $\Omega^{g_0}$ at the initial moment $t=0$ is Griffiths non-negative. Then for any $t_0\in(0,\tau)$ the set
\[
Z=\{(\xi,\eta) |\ \xi,\eta\in T^{1,0}M,\ \Omega^{g(t_0)}(\xi,\bar\xi,\eta,\bar\eta)=0\}\subset T^{1,0}M\oplus T^{1,0}M
\]
is invariant under the torsion-twisted parallel transport~\eqref{eq:parallel_transport}.
\end{theorem}

\begin{proof}
Consider $P=M\times (0,\tau)$ and define a fiber bundle $\pi\colon V\to P$, where
\[
\pi^{-1}(x,t)=\{(\xi,\eta) |\ \xi,\eta\in T_x^{1,0}M,\ \xi,\eta\neq 0\}.
\] 
We treat $\Omega$ as a function on $V$, namely, the value of $\Omega$ at $(x,t,\xi,\eta)\in V$ for $x\in M$, $t\in (0;\tau)$, and $(\xi, \eta)\in T_x^{1,0}M\oplus T_x^{1,0}M$ is defined to be $\Omega^{g(t)}_x(\xi,\bar{\xi},\eta,\bar{\eta})$.

By Theorem~\ref{thm:main} for any $t_0\in (0,\tau)$ curvature $\Omega^{g(t_0)}$ is Griffiths non-negative. Therefore $\Omega$ defines a non-negative function on $V$. We identify the zero locus $Z$ with a subset of $V$. At any $p \in V$ covariant derivative $\n^1\oplus \n^2$ defines a splitting of a tangent space to $V$ into vertical and horizontal parts
\[
T_p V=T_p^{v}V\oplus T_p^{h}V.
\]
Solutions to the parallel transport equation~\eqref{eq:parallel_transport} are precisely horizontal lifts of curves in $M\!\times\!\{t_0\}\subset P$ to $V$.

For a point $p=(x,t_0,\xi,\eta)\in Z\subset V$ over $(x,t_0)\in P$ let us pick a relatively compact neighbourhood in the total space $\mc U_p\subset V$. Fix a collection of real vector fields $v_1,\dots,v_{2n}\in \Gamma(TM)$ which are $g(t_0)$-orthonormal at $x$. Let $\widetilde v_1,\dots,\widetilde v_{2n} \in TV$ be the horizontal lifts of $v_1,\dots,v_{2n}$ to $\mc U_p$. Also let $\widetilde w_1,\dots,\widetilde w_{2n}$ be the horizontal lifts of $\n_{v_1}v_1,\dots, \n_{v_{2n}}v_{2n}$. 

Then $(\Delta^T \Omega)(p)=
\sum^{2n}_{i=1}
\bigl(
	\widetilde v_i\!\cdot\!(\widetilde v_i\!\cdot\! \Omega(p))
	-\widetilde{w_i}\!\cdot\!\Omega(p)
\bigr)$ and we can rewrite evolution equation~\eqref{eq:dt_Omega_torsion_twisted} as follows
\begin{equation}\label{eq:br_sch_ineq_1}
\frac{d}{dt} \Omega(p)=\sum^{2n}_{i=1}
\bigl(
	\widetilde v_i\!\cdot\!(\widetilde v_i\!\cdot\! \Omega(p))
	-\widetilde{w_i}\!\cdot\!\Omega(p)
\bigr) +
Q_2(\n T)(p)+
F_1(\Omega)(p)+
F_2(\Omega)(p).
\end{equation}
Clearly, $Q_2(\n T)\ge 0$ everywhere on $V$. Since neighborhood $\mc U_p$ is relatively compact, Proposition~\ref{prop:var_fiber} imply that
\[
F_2(\Omega)(p)\ge K_1\cdot\!\!\!\! \inf_{\substack{\zeta\in T^v_{p} V\\ |\zeta|\le 1}} {\delta^2} \Omega(p)(\zeta,\zeta).
\]
for some uniform in $p$ positive constant $K_1$. Also for some  positive constant $K_2$ one has
\[
\sum_{i=1}^{2n}\widetilde{w_i}\!\cdot\!\Omega(p)+\frac{d}{dt}\Omega(p)-F_1(\Omega)(p)\le K_2|{\rm grad}\,\Omega|(p),
\]
where $|{\rm grad}\,\Omega|$ is the length of the gradient of $\Omega\colon V\to\R$ with respect to any fixed axillary metric on~$V$. Then for a large enough constant $K$ we have
\begin{equation}
\sum_{i=1}^{2n}\widetilde v_i\!\cdot\!(\widetilde v_i\!\cdot\! \Omega(p)) \le 
-K\cdot\!\!\!\! \inf_{\substack{\zeta\in T^v_p V\\ |\zeta|\le 1}} {\delta^2} \Omega(p)(\zeta,\zeta)+
K|{\rm grad}\,\Omega|(p).
\end{equation}
By extending the infimum domain from $T^v_pV$ to $T_p V$ we obtain
\begin{equation}\label{eq:br_sch_ineq_crucial}
\sum_{i=1}^{2n}\widetilde v_i\!\cdot\!(\widetilde v_i\!\cdot\! \Omega(p)) \le 
-K\cdot\!\!\!\! \inf_{\substack{\zeta\in T_p V\\ |\zeta|\le 1}} {\delta^2} \Omega(p)(\zeta,\zeta)+
K|{\rm grad}\,\Omega|(p).
\end{equation}

Now we are in position to use the following result.
\begin{proposition}\cite[Prop.\,4]{br-sc-08}
Let $\mc V$ be an open subset of $\R^n$, and let $X_1,\dots, X_m$ be smooth vector fields on $\mc V$. Assume that $u\colon \mc V\to \R$ is a non-negative smooth function satisfying
\[
\sum_{j=1}^m X_j\!\cdot\! (X_j\!\cdot\! u)\le
-K \inf_{|\xi|\le 1} (\delta^2 u)(\xi, \xi)+K|Du|+Ku,
\]
where $K$ is a positive constant. Let $F =\{x\in\mc V:u(x)=0\}$. Finally, let $\gamma: [0, 1]\to\mc V$ be a smooth path such that $\gamma(0)\in F$ and $\gamma'(s)=\sum_{j=1}^m f_j(s)X_j(\gamma(s))$, where $f_1,\dots,f_m\colon [0,1]\to \R$ are smooth functions. Then $\gamma(s)\in F$ for all $s\in [0, 1]$.
\end{proposition}

We apply this proposition to the function $\Omega$ restricted to $\mc U_p\subset V$. With the use of
inequality~\eqref{eq:br_sch_ineq_crucial} we conclude that the zero set $Z\cap\mc U_p$ is invariant under the flow generated by the vector fields $\widetilde v_1,\dots, \widetilde v_{2n}$. It remains to notice that the integral curves of this flow coincide with the solutions to the torsion-twisted parallel transport~\eqref{eq:parallel_transport} along the integral curves of $v_1,\dots,v_{2n}$.
\end{proof}
Theorem~\ref{thm:brendle_schoen} implies a more standard version of the strong maximum principle.
\begin{corollary}\label{cor:smp}
Let $g(t), t\in[0,\tau)$ be a solution to the HCF~\eqref{eq:HCF} on a compact complex Hermitian manifold $(M,g_0,J)$. Assume that the Chern curvature $\Omega^{g_0}$ at the initial moment $t=0$ is Griffiths non-negative everywhere on $M$ and Griffiths positive at some point $x\in X$. Then $\Omega^{g(t)}$ is Griffiths positive everywhere on $M$ for any $t\in(0,\tau)$.
\end{corollary}
\begin{proof}
Since $\Omega^{g_0}_x$ is Griffiths positive, for any small enough $t_0\in(0,\tau)$ curvature $\Omega^{g(t_0)}_x$ is positive. So the zero set of $\Omega^{g(t_0)}$ at $x$
\[
Z_x=\{(\xi,\eta) |\ \xi,\eta\in T_x^{1,0}M,\ \Omega^{g(t_0)}(\xi,\bar\xi,\eta,\bar\eta)=0\}\subset Z,
\]
is empty. On the other hand by Theorem~\ref{thm:brendle_schoen} any two such zero sets $Z_{x_1}$, $Z_{x_2}$ could be identified via the torsion-twisted parallel transport along a curve joining $x_1,x_2\in M$ (of course this identification depends on a curve). Hence all $Z_y$, $y\in M$ are empty and $\Omega^{g(t_0)}$ is Griffiths positive. Then Theorem~\ref{thm:main} implies that $\Omega^{g(t)}$ remains positive for any $t\in (t_0,\tau)$. Sending $t_0$ to 0 we get the stated positivity for any $t\in(0,\tau)$.
\end{proof}

\section{Applications}\label{sec:applications}

In this section we demonstrate how regularization properties of the HCF could be used to study compact complex Hermitian manifolds with non-negative Griffiths curvature. Conjecturally, all such manifolds which are Fano are isomorphic to the generalized flag manifolds $G/P$, i.e., quotients of semisimple Lie groups by parabolic subgroups. 

To be more precise, Campana-Peternell conjecture~\cite{pe-ca-91} states that any Fano manifold with \emph{nef} (numerically effective) tangent bundle (see~\cite{dps-94} for the definitions and examples) is isomorphic to a generalized flag manifold. It is well-known that existence of a Hermitian metric with Griffiths non-negative curvature implies nefness. Hence Campana-Peternell conjecture would imply characterization of Fano manifolds admitting Griffiths non-negative Hermitian metric. We expect that the HCF might be a useful tool in studying this weaker version of Campana-Peternell conjecture.

First we use Corollary~\ref{cor:smp} together with the celebrated Mori's solution to the Frankel conjecture~\cite{mo-79} to prove the following uniformization result.
\begin{proposition}\label{prop:pn}
Let $(M,g_0,J)$ be a compact complex $n$-dimensional Hermitian manifold such that
\begin{enumerate}
\item its Chern curvature $\Omega$ is Griffiths non-negative;
\item $\Omega$ is strictly positive at some point $x_0\in M$.
\end{enumerate}
Then $M$ is biholomorphic to the projective space $\P^n$.
\end{proposition}
\begin{proof}
Let $g(t)$ be a solution to the HCF on $(M,g_0,J)$. By Corollary~\ref{cor:smp}, for any $t>0$ metric $g(t)$ has Griffiths-positive Chern curvature. In particular, the first Chern-Ricci form
\[
\omega_{i\bar j}:=\frac{i}{2\pi}g^{k\bar l}\Omega_{i\bar j k\bar l},
\]
representing the first Chern class of the anticanonical bundle, is strictly positive. Therefore $-K_M=\Lambda^n T^{1,0}M$ is ample and $M$ is projective.

On the other hand, it is well-known that strict Griffiths-positivity of $\Omega$, implies ampleness of $T^{1,0}M$. Indeed, if $\Omega$ is positive, then the induced metric on the line bundle $\mc L=\mc O_{\P(T^{1,0}M)}(1)$ over the projectivized manifold of hyperplanes of $T^{1,0}M$ is positive, see~\cite[Prop.\,9.1]{gr-65}, hence, by the result of Hartshorne~\cite[Prop.\,3.2]{ha-66}, $T^{1,0}M$ is ample.

Finally, using the result of Mori~\cite{mo-79}, we conclude that $M$ is isomorphic to the projective space~$\P^n$.
\end{proof}

Proposition~\ref{prop:pn} demonstrates that under the HCF either $\Omega^{g(t)}$ immediately becomes positive and in this case $M\simeq \P^n$, or $\Omega^{g_0}$ has zeros everywhere on $M$. In the next theorem we study the structure of the zero set of $\Omega^{g(t)}$ for $t>0$. It turns out, that even though we do not put any restrictions on $\Omega^{g_0}$ besides Griffiths non-negativity, zero locus of $\Omega^{g(t)}$ for any positive $t$ possesses many nice additional properties, which are not a priori satisfied.
\begin{theorem}\label{thm:regularization}
Let $g(t), t\in[0,\tau)$ be a solution to the HCF~\eqref{eq:HCF} on a compact complex Hermitian manifold $(M,g_0,J)$. Assume that the Chern curvature $\Omega^{g_0}$ is Griffiths non-negative. Fix $t_0\in(0,\tau)$ and assume that $\Omega_x^{g(t_0)}(\xi,\bar\xi,\eta,\bar\eta)=0$ for some $\xi,\eta\in T_x^{1,0} M$. Then for any $\zeta,\nu\in T_x^{1,0}M$ one has
\begin{enumerate}
\item $\Omega_x(\xi,\bar\zeta,\nu,\bar\eta)=0$;
\item $\sum_{i,j}\Omega_x(\xi,\bar\xi,e_i,\bar e_j)\Omega_x(e_j,\bar e_i,\eta,\bar\eta)=
\sum_{i,j}\Omega_x(e_i,\bar\xi,e_j,\bar\eta)\Omega_x(\xi,\bar e_j,\eta,\bar e_i)$;
\item $g(\n_\xi T_x(\zeta,\nu),\bar \eta)=0$,
\end{enumerate}
where $\Omega=\Omega^{g(t_0)}$, $T=T^{g(t_0)}$ and $e_1,\dots,e_n\in T_x^{1,0}M$ is an orthonormal frame.
\end{theorem}
\begin{proof}
By Theorem~\ref{thm:main} Chern curvature  $\Omega^{g(t)}$ is Griffiths non-negative for any $t\in [0,\tau)$. This implies that
\[
\frac{d(\Omega^{g(t)}(\xi,\bar{\xi},\eta,\bar{\eta}))}{dt}\Bigr|_{t=t_0}=0.
\]
On the other hand by Proposition~\ref{prop:dt_Omega_torsion_twisted}
\begin{equation}
\begin{split}
\frac{d(\Omega^{g(t)}(\xi,\bar{\xi},\eta,\bar{\eta}))}{dt}\Bigr|_{t=t_0}=
\Delta^T\Omega(\xi,\bar{\xi},\eta,\bar{\eta})+
Q_2(\n T)(\xi,\bar{\xi},\eta, \bar{\eta})+
F_2(\Omega)(\xi,\bar{\xi},\eta,\bar{\eta})+
F_1(\Omega)(\xi,\bar{\xi},\eta,\bar{\eta}).
\end{split}
\end{equation}
Similarly to the proof of Theorem~\ref{thm:main} at time $t_0$ we have
\begin{itemize}
\item $\Delta^T\Omega_x(\xi,\bar{\xi},\eta,\bar{\eta})\ge 0$, since $\Omega$ attains its local minimum at $x\in M$, $\xi,\eta\in T_x^{1,0}M$;
\item $Q_2(\n T)(\xi,\bar{\xi},\eta,\bar{\eta})=
\frac{1}{2}\sum\limits_{i,j}|g(\n_{\xi}T(e_i,e_j),\bar{\eta})|^2\ge 0$, see~\eqref{eq:Q_2_nT};
\item $F_2(\Omega)(\xi,\bar{\xi},\eta,\bar{\eta})$ is the sum of two non-negative summands
$\sum\limits_{i,j}
\bigl(
	\Omega_x(\xi,\bar\xi,e_i,\bar e_j)\Omega_x(e_j,\bar e_i,\eta,\bar\eta)-
	\Omega_x(e_i,\bar\xi,e_j,\bar\eta)\Omega_x(\xi,\bar e_j,\eta,\bar e_i)
\bigr)$ and
$\sum\limits_{i,j}|\Omega(\xi,\bar{e_i},e_j,\bar{\eta})|^2$;
\item $F_1(\Omega)(\xi,\bar{\xi},\eta,\bar{\eta})=0$ as a first variation of curvature at its minimum.
\end{itemize}
The sum of several non-negative summands is zero, so each summand vanishes. This implies the stated identities.
\end{proof}

\bibliographystyle{siam}
\bibliography{biblio}

\end{document}